\documentclass[12pt]{amsart}

\usepackage{amsfonts, amssymb, amsmath, latexsym,  mathtools, amscd}
\usepackage{hyperref,graphicx}
\usepackage{enumerate}
\usepackage[dvipsnames]{xcolor}
\usepackage{url}
\usepackage{wrapfig}

\usepackage{vmargin}
\setpapersize{USletter}
\setmargrb{2.5cm}{2cm}{2.5cm}{2.cm} 
\newtheorem{Theorem}{Theorem}[section]
\newtheorem{Proposition}[Theorem]{Proposition}

\newtheorem{Lemma}[Theorem]{Lemma}
\newtheorem{Corollary}[Theorem]{Corollary}

\theoremstyle{remark}
\newtheorem{remark}[Theorem]{Remark}
\newtheorem{example}[Theorem]{Example}
\newtheorem{definition}[Theorem]{Definition}
\newtheorem{defConstr}[Theorem]{Definition-Construction}

\newenvironment{Remark}
{\pushQED{\qed}\remark}
{\popQED\endremark}
\newenvironment{Example}
{\pushQED{\qed}\example}
{\popQED\endexample}
\newenvironment{Definition}
{\pushQED{\qed}\definition}
{\popQED\enddefinition}
\newenvironment{DefConstr}
{\pushQED{\qed}\defConstr}
{\popQED\enddefConstr}

\newcommand{\CC}{{\mathbb C}}
\newcommand{\RR}{{\mathbb R}}
\newcommand{\TT}{{\mathbb T}}
\newcommand{\ZZ}{{\mathbb Z}}

\newcommand{\calE}{{\mathcal E}}
\newcommand{\calV}{{\mathcal V}}
\newcommand{\BV}{{\textit{BV}}\hspace{0.5pt}}

\newcommand{\defcolor}[1]{{\color{blue}#1}}
\newcommand{\demph}[1]{\defcolor{{\sl #1}}}

\definecolor{TAMU}{RGB}{140,0,0}
\definecolor{myblue}{RGB}{0,0,198}
\definecolor{myred}{RGB}{182,0,0}

\title{The Spectral Edges Conjecture via Corners}

\author{M.~Faust}
\address{Matthew Faust, Department of Mathematics,
         Michigan State University, East Lansing, MI 48832, USA}
\email{mfaust@msu.edu}
\urladdr{https://mattfaust.github.io/}
\author{F. Sottile}
\address{Frank Sottile, Department of Mathematics,
         Texas A\&M University, College Station, Texas 77843,  USA}
\email{sottile@tamu.edu}
\urladdr{https://franksottile.github.io}

\thanks{Research supported in part by NSF grants DMS-2201005 and DMS-2052519, and Simons grant TSM-00014009.}
\begin{document}

\begin{abstract}
 The Spectral Edges Conjecture is a well-known and widely believed conjecture in the theory of discrete periodic
 operators.
 It states that the extrema of the dispersion relation are isolated, non-degenerate, and occur in a single band.
 We present two infinite families of periodic graphs which satisfy the Spectral Edges Conjecture.
 For each, every extremum of the dispersion relation is a corner point (point of symmetry).
 In fact, each spectral band function is a perfect Morse function.
 We also give a construction that increases dimension, while preserving that each spectral band function is a perfect Morse
 function.
\end{abstract}
\maketitle 

\section*{Introduction}

A periodic graph is a discrete model of a crystal.
In this setting, evolution equations governing electron transport become operators on the graph.
Floquet theory~\cite{KuchBook} realizes the spectrum of a periodic operator as images of spectral
band functions, defined either on a Brillouin zone or a compact torus.
The union of their graphs is the dispersion relation.
The structure of the band edges is of interest in mathematical physics; it is widely assumed that for generic operators, above
the band edges, critical points of the corresponding band function are nondegenerate.
This Spectral Edges Conjecture~\cite[Conj.\ 5.25]{KuchBAMS} holds in dimension 1~\cite[\S XIII.16]{RS78},
but can fail for Schr\"odinger operators~\cite{fk18} in higher dimensions.

We give two families of periodic graphs--infinitely many in every dimension---for which the
Spectral Edges Conjecture holds for each graph.
For this, we show that the spectral band functions have the significantly stronger property of being
perfect Morse functions.
We also give a construction that increases the dimension, while preserving nondegeneracy of spectral edges.
We use algebraic methods, as Floquet theory transforms analytic questions in spectral theory into algebraic questions.
Algebraic aspects of periodic operators are developed in the survey~\cite{AAPGO}, and have been invaluable in
recent  results~\cite{FL-G,FLM22,GKT,Liu22}.

In Section~\ref{S:Sparse} we show that if the dispersion polynomial is sufficiently sparse, then the spectral edges conjecture holds.
In Section~\ref{S:Isthmus}, we identify a structure of the graph $\Gamma$ which implies that the spectral edges conjecture
holds for discrete periodic Schr\"odinger operators.
Finally, in Section~\ref{Sec:Parallel} we give a construction of a $\ZZ^{d+1}$-periodic graph from a $\ZZ^d$-periodic
graph which preserves the degeneracy/nondegeneracy of critical points of spectral band functions.

These results use algebra and global arguments and are distinct from, but related to those in~\cite{BCCM},
which uses analysis to give a local criterion for a  local extremum of a spectral band function to be a global extremum.
They give this for graphs which are minimally connected, similar to the graphs we consider in Section~\ref{S:Isthmus}.
The graphs in Section~\ref{S:Sparse} are not necessarily minimally connected in that sense (e.g.~Example~\ref{Ex:singularHouse}).
As the spectral band functions we study are perfect Morse functions, local extrema are global extrema, as in~\cite{BCCM}.
Finally, both papers consider critical points at points of symmetry, called corner points.

\section{Background}
Let $\RR$, $\CC$, $\TT$, $\CC^\times$, and $\ZZ$ be, respectively, the real numbers, complex numbers,
the unit complex numbers, the nonzero complex numbers, and the integers.
We let $d>0$ be a positive integer, which may be called dimension, and $e_1,\dotsc,e_d\in\ZZ^d\subset\RR^d$
be the standard basis vectors.
We sketch some background on periodic graph operators and Floquet theory.
For more, see any of~\cite{AshcroftMermin,BerkKuch,Kittel,KorotyaevSabruova14,KuchBook,AAPGO}.

Let \defcolor{$\Gamma$} be a graph equipped with a free cocompact action of $\ZZ^d$.
That is,  $\ZZ^d$ has finitely many orbits on the vertices \defcolor{$\calV$} and edges
\defcolor{$\calE$} of $\Gamma$.
Write the action of $a\in \ZZ^d$ on $v\in \calV$ as $a{+}v$.
Let $\defcolor{W}\subset\calV$ be a collection of orbit representatives, called a \demph{fundamental domain}.

A \demph{labeling} of $\Gamma$ is a pair of  $\ZZ^d$-periodic functions
$\defcolor{V}\colon \calV \to \RR$ and $\defcolor{E}\colon \calE \to \RR$.
We write $(\Gamma,V,E)$ for the resulting labeled periodic graph.
These are the data for a \demph{discrete periodic operator} on $\Gamma$.
We let $\defcolor{H}\vcentcolon= V+\Delta_E$ be the sum of the multiplication operator $V$ and a 
weighted adjacency operator $\defcolor{\Delta_E}$  acting on complex-valued functions on $\calV$.
Given a function $f\colon \calV\to \CC$, the value of $Hf$ at a vertex $v\in\calV$ is
 \begin{equation}\label{eq1}
     (Hf)(v)\ \vcentcolon=\ V(v) f(v) + \sum_{(v,u) \in \calE} E(v,u)f(u)\,.
 \end{equation}
For the constant function $E\colon \calE \to \{1\}$, $H$ is the
\demph{discrete periodic Schr\"odinger operator}.
By cocompactness, $H$ is a bounded operator on the Hilbert space $\ell_2(\calV)$ of square-summable functions on
$\calV$.
As $V$ and $E$ are real-valued, it is self-adjoint and has real spectrum $\defcolor{\sigma(H)}$.

Floquet theory describes the interaction of the spectrum with the representations of $\ZZ^d$,
leading to the dispersion relation between the energy $\lambda$ and quasi-momenta $k\in\RR^d$.
A \demph{Floquet function}  with quasi-momentum $k$ is a function $g$ on $\calV$ satisfying
\[
g(a{+}v)\ =\ e^{2\pi\sqrt{-1} k\cdot a}g(v) \qquad\mbox{ for all } a\in\ZZ^d\ \mbox{ and }v\in\calV\,.
\]
As the topology of the compact torus $\TT^d$ plays a role, we change variables.
Set $\defcolor{z}\vcentcolon= e^{2\pi\sqrt{-1} k}\in\TT^d$, which we call a  \demph{Floquet multiplier}.
Then $e^{2\pi\sqrt{-1} k\cdot a} = \defcolor{z^a} \vcentcolon= z_1^{a_1}\dotsb z_d^{a_d}$,
where $z=(z_1,\dotsc,z_d)$, and  quasi-periodicity of $g$ becomes
\begin{equation}\label{eq2}
    g(a{+}v)\ =\ z^a  g(v)\,,\ \text{ for all } v \in \calV\  \text{and } a\in \ZZ^d\,.
\end{equation}
A Floquet function $g$ with multiplier $z$ is determined by its values on the fundamental domain $W$.
On such a function, the operator $H$ becomes
\begin{equation}\label{eq3}
     (Hg)(v)\ =\  V(v) g(v) + \sum_{(v,a+u) \in \calE} E(v,a{+}u)  z^ag(u)\,,
\end{equation} 
where $u,v \in W$ and $a \in \ZZ^d$.
Thus, $H$ acts as multiplication by a $|W| \times |W|$ matrix $H(z)$.
The spectrum $\defcolor{\sigma_z(H)}$ of $H(z)$ consists of the roots of its characteristic polynomial.

By Floquet theory the spectrum $\sigma(H)$ is the union of these eigenvalues $\sigma_z(H)$,
\begin{equation}\label{Eq:Floquet}
     \sigma(H)\ =\ \bigcup_{z\in\TT^d} \sigma_z(H) 
    \ =\ \{ \lambda\in\RR \mid \exists z\in\TT^d\;\mbox{ s.t.~}\; \det(H(z)-\lambda)=0\}\,.
 \end{equation}      
This expression for the spectrum leads to the following interpretation.

Treating the coordinates of $z=(z_1,\dotsc,z_d)\in\TT^d$ as indeterminates,  $H(z)$ becomes a $|W|\times|W|$ matrix
with Laurent polynomial entries. 
The entry in row $v$ and column $u$ is
 \begin{equation}\label{Eq:MatrixEntry}
    V(v)\delta_{v,u}\ +\  \sum_{(v,a+u) \in \calE} E(v,a{+}u)  z^a\qquad \in\ \RR[z_1^{\pm}, \dots, z_d^{\pm}]\,.
 \end{equation}
We call this matrix \defcolor{$H(z)$} of Laurent polynomials the \demph{Floquet matrix} and refer to its
characteristic polynomial $D(z,\lambda)\vcentcolon= \det(H(z)-\lambda)$ as the \demph{dispersion polynomial}.

The \demph{dispersion relation}, \defcolor{$\mbox{DR}$}, is the subset of $\TT^d\times\RR$ defined by the vanishing of the dispersion polynomial.
That is,
\[
\mbox{DR}\ \vcentcolon=\ \{ (z,\lambda)\in\TT^d\times\RR \mid D(z,\lambda)\ =\ 0\}\,.
\]
By~\eqref{Eq:Floquet}, its projection to $\RR$ (the $\lambda$ axis) is the spectrum of $H$.
By periodicity, if $(v,a{+}u)$ is an edge of $\Gamma$, then so is $(u, -a{+}v)$.
Thus, \eqref{Eq:MatrixEntry} implies that $H(z)^T=H(z^{-1})$, 
and for $z\in\TT^d$, $H(z)$ is hermitian so it has $|W|$ real eigenvalues.
Thus, the dispersion relation consists of $|W|$ branches, each of which is the graph of 
a \demph{spectral band function}.
Figure~\ref{F:DispersionRelations} shows three dispersion relations for $d=2$,
\begin{figure}[htb]
   \centering
   \includegraphics[height=100pt]{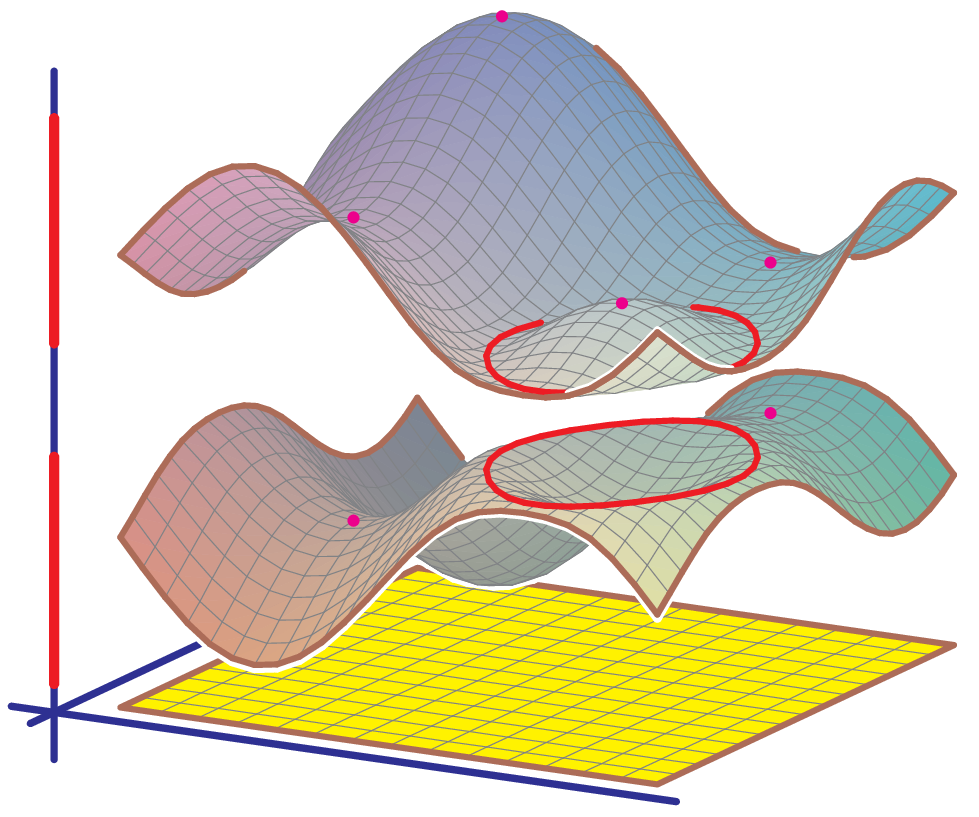}\qquad
   \includegraphics[height=110pt]{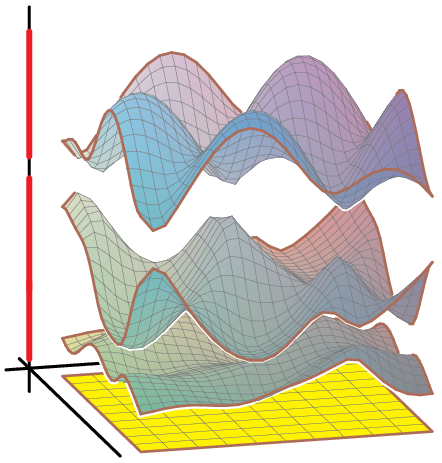}\qquad
   \includegraphics[height=120pt]{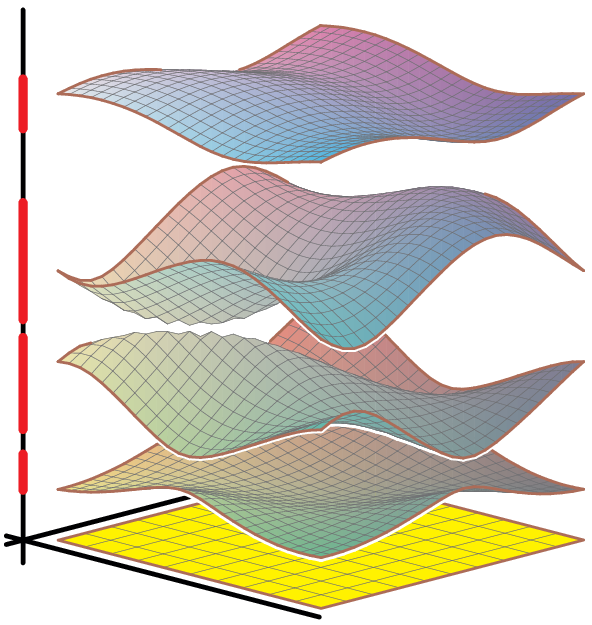}
     
  \caption{Three dispersion relations for graphs with $d=2$ and $|W|=2,3,4$.}
  \label{F:DispersionRelations}
\end{figure}  
each drawn over $[-\frac{\pi}{2},\frac{3\pi}{2}]^2$, which is
a Brillouin zone, equivalently, a fundamental domain for the torus $\TT^2$.
The image in $\RR$ of each spectral band function is a \demph{spectral band}.
These may overlap, as in the middle dispersion relation.

In what follows, \demph{generic} means on the complement of a proper algebraic subset of $\RR^n$ or $\CC^n$.
A generic set in particular is open and dense.
The set of parameters for $\Gamma$ forms a finite-dimensional vector space whose dimension is the 
number of orbits of $\ZZ^d$ on $\calV\cup\calE$.
The \demph{spectral edges conjecture} for a graph $\Gamma$ states that, for a generic choice of $V$ and $E$,
the extrema of the dispersion relation are isolated, nondegenerate, and each occurs on a single spectral band function.
The conjecture was first made for discrete periodic Schr\"odinger operators and was believed to be true for any
connected graph; however, the conjecture was disproven in \cite{fk18}.
An example of this is shown  on the left in Figure~\ref{F:DispersionRelations}, each spectral band function has a
curve of critical points.
Although the conjecture fails for Schr\"odinger operators, it remains open for discrete periodic operators.
It is also of interest to understand for which graphs the conjecture holds for Schr\"odinger operators. 

The most extreme example of non-isolated critical points is when a spectral band function is a constant,
$\lambda_0$, so that $D(z,\lambda_0)=0$.
Equivalently, when $H$ has eigenvectors with eigenvalue $\lambda_0$, necessarily of infinite
multiplicity~\cite[\S~3.4.1]{AAPGO}.
When this occurs, the dispersion relation has a \demph{flat band}.
Figure~\ref{F:Lieb} shows the $\ZZ^2$-periodic Lieb lattice, 
\begin{figure}[htb]
\centering

\includegraphics[height=90pt]{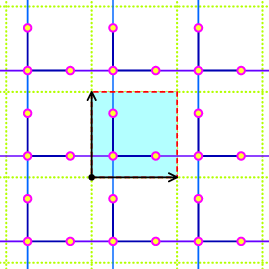}\qquad
\includegraphics[height=100pt]{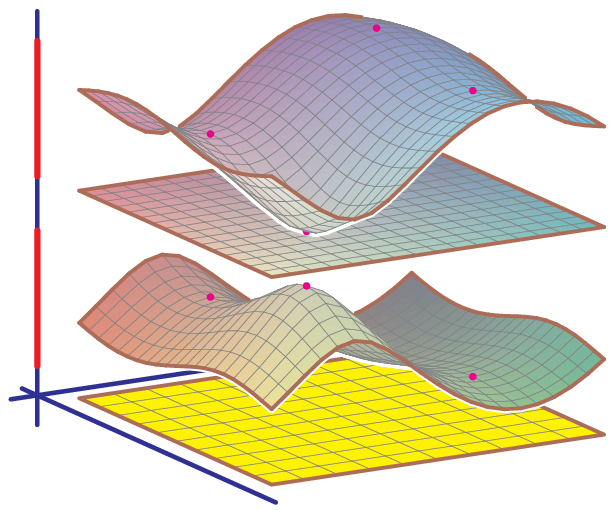}\qquad
\includegraphics[height=110pt]{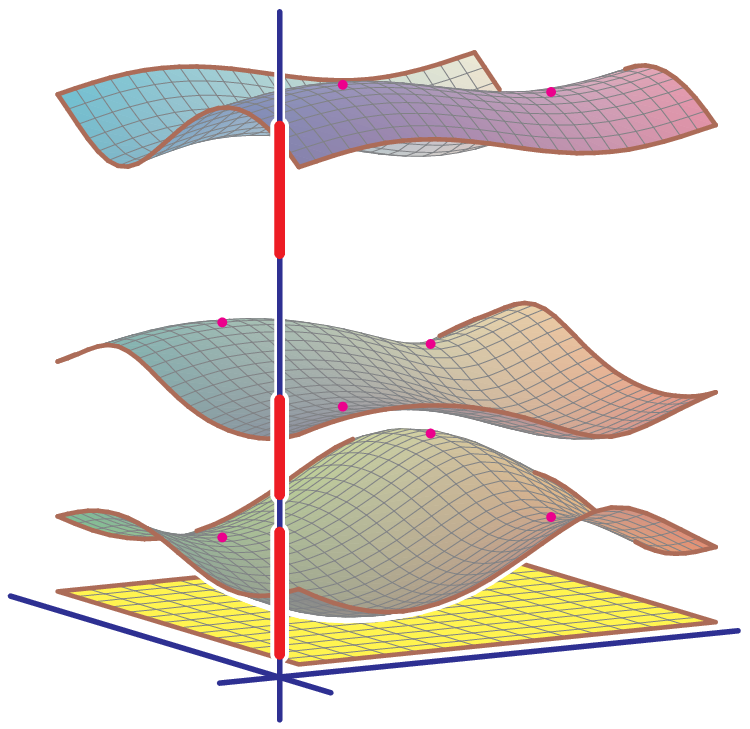}

\caption{The Lieb lattice and two of its  dispersion relations.}
\label{F:Lieb}
\end{figure}
with two dispersion relations for operators on it.
The one on the left has a flat band.
The main result of~\cite{faust2025rareflatbandsperiodic} is that for a given graph, the set of parameters for which
the dispersion relation has a flat band lies in a proper algebraic subset.
This is strengthened in~\cite{FaustKachkovskiy}.

We present two families of $\ZZ^d$-periodic graphs, each having infinitely many members for every $d>0$, such that
every graph in each family satisfies the spectral edges conjecture.
Graphs in the first family satisfy the conjecture for discrete periodic operators, while those in the second family
satisfy the conjecture for discrete periodic Schr\"odinger operators.

For graphs in these two families, we establish the spectral edges conjecture via  a stronger result, 
the critical point conjecture.
This was formulated in~\cite{DKS} and studied in~\cite{FS24}.
For this, we complexify, extending our variables from $\TT^d\times\RR$ to $(\CC^\times)^d\times\CC$.
The complexification of the dispersion relation is the \demph{Bloch variety},
\[
   \defcolor{\BV}\ \vcentcolon=\ \{ (z,\lambda)\in(\CC^\times)^d\times\CC \mid D(z,\lambda)=0\}\,.
\]
A  point $(z,\lambda) \in (\CC^\times)^d \times \CC$ is a \demph{critical point} (of the function $\lambda$ on the
Bloch variety) if it is a solution to the system of polynomial equations,
called the \demph{critical point equations}
 \begin{equation}\label{Eq:CPE}
   D(z,\lambda)\ =\ \frac{\partial D}{\partial z_1}(z,\lambda)\ =\
   \dotsb\ =\ \frac{\partial D}{\partial z_d}(z,\lambda)\  =\ 0\,.
 \end{equation}
Solutions are points $(z,\lambda)$ on the Bloch variety where the gradient of the function $\lambda$ vanishes.
These include all non-smooth points of the Bloch variety and every extremum of a
spectral band function is a critical point.
The \demph{critical point conjecture} states that for generic parameters $(V,E)$, every critical point
$(x,\lambda_0)$ of the Bloch variety is smooth and isolated.
Equivalently, the Jacobian matrix of the critical point equations~\eqref{Eq:CPE} is invertible at $(x,\lambda_0)$.
We relate this to the spectral edges conjecture.

\begin{Theorem}\label{Th:CPC_SEC}
  The critical point conjecture implies the spectral edges conjecture.
\end{Theorem}
\begin{proof}
  Suppose that the critical point conjecture holds for a particular graph.
  Let $V$, $E$ be parameters for operators on $\Gamma$ such that all critical points are smooth and isolated, and let  
  \defcolor{$(x,\lambda_0)$} be an extremum of a spectral band function.
  Then it is a critical point and thus an isolated solution to the critical point equations~\eqref{Eq:CPE}
  with invertible Jacobian.
  By~\cite[Lem.\ 5.1]{FS24}, the  Hessian matrix at $(x,\lambda_0)$ is invertible, so that it is a
  nondegenerate extremum.

  It remains to show that each extremum occurs on a single spectral band function, for generic parameters.
  We argue that this holds when the values of the potential $V$ are distinct and the edge parameters $E$ are sufficiently small,
  which implies that this holds for general parameters.
  Suppose that the potential values $V(v)$ for $v\in W$ are distinct, and let $E\colon\calE\to\RR^\times$ be any 
  of nonzero periodic function.
  For $t\geq 0$, define $\defcolor{E_t}\colon\calE\to\RR^\times$ by $E_t(u,v)\vcentcolon= t\cdot E(u,v)$, for $(u,v)\in\calE$.
  Then the dispersion polynomial for parameters $(V,E_t)$  has the form
  \[
      D_t(z,\lambda)\ =\ \prod_{v\in W}(V(v)-\lambda)\ +\ t\cdot F(z,\lambda,t)\,,
  \]
  for some polynomial $F$, which is a Laurent polynomial in $z$ and an ordinary polynomial in $\lambda$ and $t$.
  When $t=0$, the dispersion relation consists of $|W|$ isolated flat bands.
  Each is the graph of a spectral band function that is constant and equal to a potential value.
  By continuity, there exists $\epsilon>0$ so that if $0<t<\epsilon$, the images of the spectral band functions remain
  disjoint, which implies that each extremum occurs for a single spectral band function.
\end{proof}

The second half of the proof is well-known folklore.
Another bit of folklore\footnote{For example, see page 4 of~\cite{ABG}.}, which we give for completeness, is that
there are at least $2^d|W|$ critical points, counted with multiplicity.
These are critical points which occur at points of symmetry or \demph{corner points}, where $z_i = \pm 1$
for each $i \in [d]$, equivalently, where $z^2=1$.

A function $f\colon M\to \RR$ on a compact manifold $M$ is \demph{Morse}~\cite{Milnor} if all of its critical points are nondegenerate.
The Morse inequality implies that the sum of the Betti numbers of $M$ is a lower bound for the number of critical points.
When these two quantities are equal,  $f$ is a \demph{perfect Morse function}.

\begin{Lemma}\label{L:minCritPts}
  Every point $(z,\lambda)$ on the Bloch variety with $z^2=1$ is critical.
  If the critical points are isolated, then there are at least $2^d|W|$ critical points, counting multiplicity.

  If there are exactly $2^d|W|$ critical points, each nondegenerate, 
  then  they all occur at the corner points and every spectral band function is a perfect Morse function.
\end{Lemma}
\begin{proof}
  As $D(z,\lambda) = D(z^{-1},\lambda)$, we have
  \[
  \frac{\partial D}{\partial z_i}(z_1,\dots,z_i,\dots, z_d)\ =\
  -z_i^{-2}\frac{\partial D}{\partial z_i}(z_1^{-1},\dots,z_i^{-1},\dots, z_d^{-1})\,.
  \]
  When $z_i = \pm 1$, this implies that $\partial D/\partial z_i=-\partial D/\partial z_i=0$.
  This implies the first statement.

  For the second, there are $2^d|W|$ points, counted with multiplicity, on the Bloch variety with $z^2=1$.
  The multiplicity of a critical point, {\sl as a point on the Bloch variety}, is a lower bound for its multiplicity
  as a critical point,
  e.g.\ as a solution to the critical point equations~\eqref{Eq:CPE}.
  This is because the multiplicity of a point $(x,\lambda_0)$ on the Bloch variety is the dimension of the
  quotient of the ring $\CC[z^{\pm},\lambda]$ localized at $(x,\lambda_0)$ by the
  ideal generated by $D$ and {\sl all} of its partial derivatives, while for the multiplicity of a solution
  to~\eqref{Eq:CPE}, we use the ideal generated by $D$ and
  {\sl only} its partial derivatives with respect to the $z$-variables.

  For the third statement, observe that each spectral band function is a continuous function  $f\colon\TT^d\to\RR$.
  If its critical points are nondegenerate, then by the Morse inequality,
  \[
  \mbox{the number of critical points  of $f$}\ \geq\ 2^d\ =\ \mbox{sum of Betti numbers of $\TT^d$}\,.
  \]
  As there are $|W|$ spectral band functions, we must have equality for every spectral band function, which implies that each
  is a perfect Morse function. 
\end{proof}

The unifying theme of the two families we will discuss is that general operators in these families
have critical points only at the corner points, all critical points are nondegenerate, and 
each spectral band function is a perfect Morse function.

Finding these families was surprising to us for the following reason.
The characteristic matrix $H(z)-\lambda$ is a map from the
$(d{+}1)$-dimensional space $(\CC^\times)^d\times\CC$ to the space \defcolor{$M_W$} of $|W|\times|W|$ matrices.
The Bloch variety is the inverse image of the determinant hypersurface, $\defcolor{\mbox{Det}}\subset M_W$.
Since the singular locus of $\mbox{Det}$ has codimension 4 in $M_W$, we expect that when $d\geq 3$, the Bloch variety will
have singularities,  which are non-isolated when $d>3$.
Any singularity is a critical point.
As observed in the proof of Lemma~\ref{L:minCritPts}, an isolated singularity
has multiplicity greater than 1 as a critical point.

Thus, when $d>2$, we expect that there are degenerate critical points.
While working on~\cite{FRS2025}, we discovered a family of $\ZZ^2$-periodic graphs whose spectral band functions were
perfect Morse functions, and realized how to generalize that to all $d$.

\section{Minimally Sparse Graphs}\label{S:Sparse}

If a $\ZZ^d$-periodic graph $\Gamma$ is connected, then the exponents $\alpha$ of $z$ which occur in its Floquet matrix $H(z)$
span $\ZZ^d$, and we expect that this also holds for its dispersion polynomial $D(z,\lambda)$.
A polynomial $D(z,\lambda)$ is \demph{minimally sparse} if the only monomials in $z$ which occur are $z_i^{\pm1}$.
A $\ZZ^d$-periodic graph $\Gamma$ is \demph{minimally sparse} if for generic parameters, its dispersion  polynomial
is minimally sparse.

Below is a generic labeling of Lieb lattice and its dispersion polynomial, where $u$, $v$, and $w$ are values of
the potential at eponymous vertices, showing that it is minimally sparse. 
\[
   \begin{picture}(104,104)(12,12)
     \put( 12, 12){\includegraphics{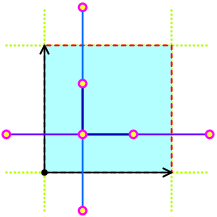}}
     \put( 54, 23){\small$c$}
     \put( 54, 65){\small$a$}
     \put( 54,101){\small$c$}
     \put( 22, 55){\small$d$}
     \put( 65, 55){\small$b$}
     \put(101, 55){\small$d$}
     \put( 40, 74){\small$w$}
     \put( 42, 42){\small$u$}     \put( 75, 42){\small$v$}
   \end{picture}\qquad
    \raisebox{65pt}{\begin{minipage}[t]{250pt}
        $\quad\,  (\lambda-u)(\lambda-v)(\lambda-w)$\vspace{3pt} \newline
        $-\lambda( a^2 + b^2 +c^2 + d^2 ) \ +\ v(a^2 + c^2) + w(b^2 + d^2)$\vspace{3pt}\newline
        $ -\ bd(z_1+z_1^{-1}) (\lambda-w) \ -\ ac(z_2+z_2^{-1})(\lambda-v)$\,.
   \end{minipage}}
\]
%

%
%

\subsection{Coordinate Projections}\label{Sec:CoordinateProjections}
The construction of the Floquet matrix from a labeled periodic graph may be reversed.
While  explained in~\cite[\S\ 3.1.2]{AAPGO}, we sketch that here.

\begin{DefConstr}\label{D:GraphFromMatrix}
  Suppose that $H(z)$ is an $m\times m$ matrix whose entries are Laurent polynomials in $z=(z_1,\dotsc,z_r)$, and which 
  has the symmetry $H(z)^T=H(z^{-1})$.
  Let \defcolor{$W$} be a set with $m$ elements $v_1,\dotsc,v_m$ with $v_i$ corresponding to the $i$th row/column of $H(z)$.
  Set $\defcolor{\calV}\vcentcolon=\ZZ^r\times W$, which carries a natural action of $\ZZ^r$,
  \[
  \mbox{for $\alpha\in\ZZ^r$ and $(\beta,v)\in\ZZ^r\times W$, }\  \alpha+(\beta,v)\ \vcentcolon=\ (\alpha+\beta,v)\,.
  \]

  We define a $\ZZ^r$-periodic potential $\defcolor{V}\colon\calV\to\RR$.
  For $(\beta,v)\in \ZZ^r\times W$, set
  \[
  V(\beta,v)\ \vcentcolon=\ \mbox{constant term of $H(z)_{v,v}$}\,.
  \]
  All other terms in $H(z)$ give labeled periodic edges as follows.
  For $v,u\in W$, if $cz^\alpha$ is a term in the entry $H(z)_{v,u}$, and we do not have $v=u$ and $\alpha=0$, 
  then we have an orbit of edges
  \begin{equation}\label{E:orbitOfEdges}
    (\beta,v)\ \sim\ (\alpha+\beta,u) \qquad \mbox{ for } \beta\in\ZZ^r\,,
  \end{equation}
  with common label $c$.
  As $H(z)^T=H(z^{-1})$, there will be a corresponding term $cz^{-\alpha}$ in the entry $H(z)_{u,v}$.
  The resulting orbit of edges,
  \[
     (\beta,u)\ \sim\ (-\alpha+\beta,v) \qquad \mbox{ for } \beta\in\ZZ^r\,,
  \]
  with label $c$ is just a reparametrization of~\eqref{E:orbitOfEdges}.

  Let \defcolor{$\calE$} be the set of the edges~\eqref{E:orbitOfEdges} for terms $cz^\alpha$ in $H(z)$ that are
  not constants in diagonal entries and let $\defcolor{E}\colon\calE\to\RR$ be the constructed $\ZZ^r$-periodic function.
  Let $\Gamma=(\calV,\calE)$ be the $\ZZ^r$-periodic graph with labels $V,E$.
 \end{DefConstr}

The point of this is the following proposition, whose proof we leave to the reader.

\begin{Proposition}\label{P:GraphFromFloquet}
  If $\Gamma=(\calV,\calE)$ with $\ZZ^r$-periodic functions $V\colon\calV\to\RR$ and $E\colon\calE\to\RR$ are
  constructed following Definition-Construction~\ref{D:GraphFromMatrix} from a matrix $H(z)$ of Laurent polynomials
  satisfying $H(z)^T=H(z^{-1})$, then $H(z)$ is the Floquet matrix of the labeled graph $(\Gamma,V,E)$.
\end{Proposition}

We use Definition-Construction~\ref{D:GraphFromMatrix} to obtain a labeled periodic graph whose Floquet matrix is obtained
from a given Floquet matrix after substituting $\pm1$ for some of its variables.

\begin{Definition}\label{D:coordinateProjection}
  Let $(\Gamma,V,E)$ be a labeled $\ZZ^d$-periodic graph with Floquet matrix $H(z)$.
  Let $I=\{i_1<\dotsb<i_r\}\subset[d]$ be $r$ indices of coordinates, and for each $j\in[d]\smallsetminus I$, let
  $\varepsilon_j\in\{\pm1\}$ be a choice of sign.
  For $z=(z_1,\dotsc,z_r)\in\TT^r$ define $\defcolor{\zeta(z)}\in\TT^d$ by its coordinates
  \begin{equation}\label{Eq:coordSubst}
    \zeta(z)_j\ =\
    \left\{\begin{array}{rcl} z_s&&\mbox{if } j=i_s\in I,\\ \varepsilon_j&&\mbox{if }j\not\in I.\end{array}\right.\ 
  \end{equation}
  Set $\defcolor{H'(z)}\vcentcolon= H(\zeta(z))$.
  Informally, we set $z_j=\varepsilon_j$ for $j\not\in I$ and then relabel the remaining variables
  $z_1,\dotsc,z_r$.
  
  As $\varepsilon_j^{-1}=\varepsilon_j$, we have $H'(z)^T=H'(z^{-1})$.
  Let \defcolor{$(\Gamma',V',E')$} be the labeled $\ZZ^r$-periodic graph constructed from $H'(z)$ in
  Definition-Construction~\ref{D:GraphFromMatrix}.
  We call $(\Gamma',V',E')$ a \demph{coordinate projection} of $(\Gamma,V,E)$.
  It is possible to define this directly from $(\Gamma,V,E)$.
  We omit the details as they are not needed, but note that if any entry $(v,w)$ of $H'(z)$ vanishes, then there is no edge between
  any two translates of $v$ and $w$ in $\Gamma'$.
\end{Definition}

\begin{Theorem}\label{Thm:RealFamCross}
  Let $(\Gamma,V,E)$ be a $\ZZ^d$-periodic labeled graph whose dispersion polynomial $D(z,\lambda)$ is  minimally sparse.
  Then one of the following holds:
  (1) The dispersion relation has a flat band,
  (2) the dispersion relation of a coordinate projection has a flat band,
  or (3) the only critical points of the dispersion relation are the corner points.
\end{Theorem}

\begin{Remark}\label{Rem:FK}
  A main result of~\cite{FaustKachkovskiy} is that if the potential is generic, then
  having a coordinate projection with a flat band is equivalent to $\Gamma'$ having a bounded connected component.
\end{Remark}

\begin{proof}[Proof of Theorem~\ref{Thm:RealFamCross}]
  As $D(z,\lambda)$ is minimally sparse and $D(z,\lambda)=D(z^{-1},\lambda)$, we have 
  \begin{equation}\label{Eq:SparseDispersion}
      D(z,\lambda)\ =\ h_0(\lambda)\ +\ \sum_{i=0}^d (z_i+z_i^{-1}) h_i(\lambda)\ ,
  \end{equation}
  for some polynomials $h_1(\lambda),\dotsc, h_d(\lambda)$.
  The Bloch variety has a flat band at $\lambda=\lambda_0$ if and only if $\lambda_0$ is a common root of every $h_i$.
  Let us suppose this not the case.

  Suppose that \defcolor{$(x,\lambda_0)$} is a critical point of the Bloch variety and that $x$ is not a corner point.
  From the formula~\eqref{Eq:SparseDispersion}, for each $i=1,\dotsc,d$, we obtain
  \[
  \frac{\partial D}{\partial z_i}(x,\lambda_0) \ =\ (1-x_i^{-2}) h_i(\lambda_0)\ .
  \]
  Then for each $i\in [d]$, either $x_i^2=1$ or else $h_i(\lambda_0)=0$.

  Let $I\subsetneq[d]$ be the set of indices such that $x_i^2\neq 1$ and thus $h_i(\lambda_0)=0$.
  For $j\not\in I$, set $\varepsilon_j\vcentcolon=x_i\in\{\pm1\}$.
  Using~\eqref{Eq:SparseDispersion} and that $h_i(\lambda_0)=0$ for $i\in I$, we have
  \begin{equation}\label{Eq:sparseD}
     0\ =\ D(x,\lambda_0)\ =\ h_0(\lambda_0)\ +\ \sum_{j\not\in I} 2\varepsilon_j h_j(\lambda_0)\,.
  \end{equation}

  For $z\in\TT^r$, define $\zeta(z)\in\TT^d$ by~\eqref{Eq:coordSubst}.
  Set $H'(z)\vcentcolon= H(\zeta(z))$, the Floquet matrix for the coordinate projection $(\Gamma',V',E')$ of
  $(\Gamma,V,E)$ arising from $I,\varepsilon$.
  Is dispersion polynomial is
  \[
  D'(z,\lambda)\ =\ \det(H'(z)-\lambda)\ =\ \det( H(\zeta(z))-\lambda)\ =\ D(\zeta(z),\lambda)\,.
  \]
  The Bloch variety for $D'(z,\lambda)$ has a flat band at $\lambda=\lambda_0$.
  Indeed,
  \begin{eqnarray*}
    D'(z,\lambda_0)
    &=& D(\zeta(z),\lambda_0)
       \ =\  h_0(\lambda_0)\ +\ \sum_{i=1}^d (\zeta(z)_i+\zeta(z)_i^{-1}) h_i(\lambda_0)\\
    & =& h_0(\lambda_0)\ +\ \sum_{j\not\in I} 2\varepsilon_j h_j(\lambda_0)\ =\ 0\,, 
  \end{eqnarray*}
  as for $i\in I$, $h_i(\lambda_0)=0$.
  The last equality is~\eqref{Eq:sparseD}.
  This completes the proof.
\end{proof}
\begin{Corollary}\label{Cor:SPEC1}
    A connected minimally sparse graph satisfies the spectral edge conjecture.
\end{Corollary}
\begin{proof}
  For any connected graph $\Gamma$,~\cite[Theorem 1.1]{faust2025rareflatbandsperiodic} states that the set of parameters such
  that the Bloch
  variety has a flat band lie in a proper algebraic subset.
  A $\ZZ^d$-periodic graph $\Gamma$ has finitely many (in fact $3^d-2^d-1$) distinct coordinate projections.
  It follows that there is a dense Zariski-open subset \defcolor{$U_\Gamma$} of the space of parameters $(V,E)$ for $\Gamma$
  such that neither $\Gamma$ nor any of its coordinate projections has a flat band.

  Suppose that $\Gamma$ is minimally sparse.
  By Theorem~\ref{Thm:RealFamCross}, for $(V,E)\in U_\Gamma$ the critical points of the Bloch variety occur only at the
  corner points.
  Let $(x,\lambda_0)$ be such a corner critical point.
  We claim that it is a nondegenerate critical point.
  By~\eqref{Eq:SparseDispersion}, we have
  \[
  \frac{\partial^2 D}{\partial z_i\partial z_j}(x,\lambda_0) \ =\
  \left\{ \begin{array}{rcl} 0&&\mbox{ if } i\neq j,\\  2x_i^{-3} h_i(\lambda_0)&&\mbox{ if } i=j.\end{array}\right. 
  \]
  Thus, the Hessian matrix at $(x,\lambda_0)$ is diagonal, and its determinant vanishes if and only if
  $h_i(\lambda_0)=0$ for some $i\in[d]$.

  Suppose that $h_i(\lambda_0)=0$.
  Consider the coordinate projection given by $I=\{i\}$ and $\varepsilon_j=x_j$ 
  for $j\neq i$.
  As in the proof of Theorem~\ref{Thm:RealFamCross}, this coordinate projection has a flat band,
  contradicting that$(V,E)\in U_\Gamma$.
  Thus, all critical points are are nondegenerate,  which implies the spectral edges conjecture, by Theorem~\ref{Th:CPC_SEC}.
\end{proof}  

By  Theorem~\ref{Thm:RealFamCross}, when a dispersion polynomial is minimally sparse and a non-corner point is critical, then
it lies on a flat band in a
coordinate projection.
There are several possible configurations for this as the following example shows.

\begin{Example}\label{Ex:singularHouse}
  Figure~\ref{fig:singularHouse} shows a $\ZZ^2$-periodic graph and a general labeling 
  \begin{figure}[htb]
     \centering
    \includegraphics[height=110pt]{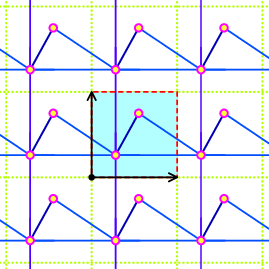}
    \qquad
   \begin{picture}(169,110)(-32,0)
    \put(  0,  0){\includegraphics{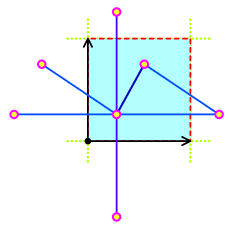}}
    \put( 36, 30){{\color{white}\circle*{9}}}
    \put( 47, 47){\small$u$}      \put( 67, 83){\small$v$}  

    \put(-32, 52){\small$-e_1{+}u$} \put(111, 52){\small$e_1{+}u$}
    \put( 61,103){\small$e_2{+}u$}  \put( 61,3){\small$-e_2{+}u$}
    \put(-18, 77){\small$-e_1{+}v$}

    \put(65, 61){\small$a$}
    \put(28, 45){\small$b$}   \put(75, 45){\small$b$}
    \put(49, 75){\small$c$}   \put(49, 24){\small$c$}
    \put(31, 73){\small$d$}   \put(82, 72){\small$d$}   
   \end{picture}

    \caption{A $\ZZ^2$-periodic minimally sparse graph.}
    \label{fig:singularHouse}
  \end{figure}
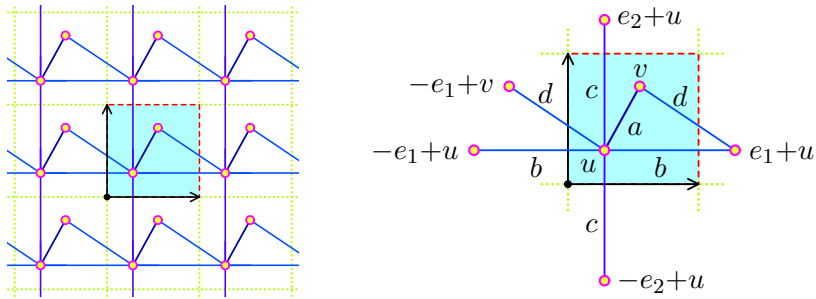
  in a neighborhood of a fundamental domain.
  We write $u$ and $v$ for the values of the potential at eponymous vertices.
  It is minimally sparse, as may be seen from its dispersion polynomial
  \[
   \lambda^2  -\lambda(u{+}v)  -a^2-d^2+uv
    + (z_1 + z_1^{-1})( b\lambda  -bv-ad)
    + (z_2 + z_2^{-1})( c\lambda -cv)\,.
  \]

  Figure~\ref{F:singHouse_BV} shows three of its Bloch varieties with their parameter values
  $(u,v,a,b,c,d)$.
  \begin{figure}[htb]

    \centering
    \begin{picture}(135,110)(0,-10)
      \put(0,0){\includegraphics[height=100pt]{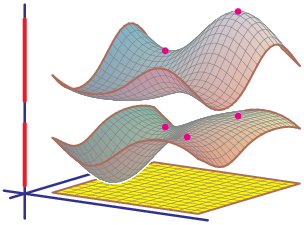}}
       \put(50,-10){\small$(3,0\,,\,3,2,2,1)$}
    \end{picture}
    \quad
    \begin{picture}(126,135)(0,-10)
      \put(0,0){\includegraphics[height=125pt]{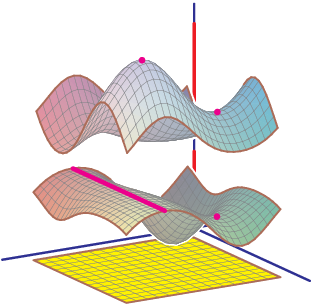}}
      \put(50,-10){\small$(2,0\,,\,1,1,1,1)$}
    \end{picture}
    \quad
    \begin{picture}(130,110)(0,-10)
      \put(0,0){\includegraphics[height=110pt]{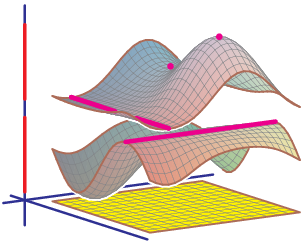}}
      \put(50,-10){\small$(1,0\,,\,1,1,1,1)$}
    \end{picture}
   \caption{Three Bloch varieties with their parameters.}\label{F:singHouse_BV}
  \end{figure}
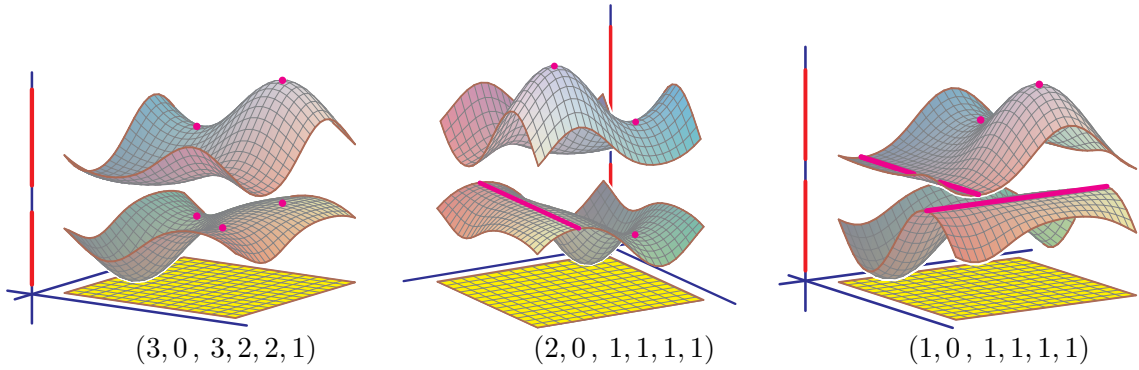
  In the first image, the only critical points are at the corner points.
  The middle image has a 1-dimensional line of critical points on its first spectral band
  function, while the third has lines of critical points on both band functions, but in different directions.
\end{Example}

\subsection{Periodic flower graphs}
Both the Lieb lattice and the graph of Example~\ref{Ex:singularHouse} are periodic flower graphs.
Every periodic flower graph is minimally sparse and there are infinitely many in every dimension $d$.

Let  $S$  be a finite connected graph with a distinguished vertex, $u$ and
let $P_1,\dotsc,P_\ell$ be  disjoint cycles of possibly varying lengths which may include loops consisting
of one vertex and one edge.
Given these, attach each cycle $P_i$ to $S$ by identifying one of its vertices with $u$, obtaining a \demph{flower} graph, $F$.
Its \demph{petals} are the cycles $P_i$.
Observe that the petals $P_i$ and $S$ only meet at the vertex $u$ and that $F$ is connected.
We show some flower graphs
\[
\begin{picture}(66,84)
  \put(0,0){\includegraphics{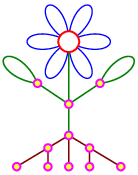}}
  \put(30,61.5){\small$u$}
\end{picture}
\qquad
\begin{picture}(64,74)
  \put(0,0){\includegraphics{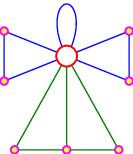}}
  \put(29,45){\small$u$}
\end{picture}
\qquad
\begin{picture}(59,74)(0,-15)
  \put(0,0){\includegraphics{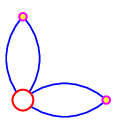}}
  \put(8,8){\small$u$}
\end{picture}
\qquad
\begin{picture}(54,74)(0,-15)
  \put(0,0){\includegraphics{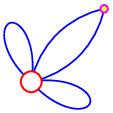}}
  \put(12,12){\small$u$}
\end{picture}
\qquad
\begin{picture}(54,74)(0,-10)
  \put(0,0){\includegraphics{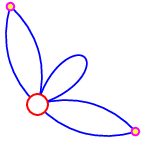}}
  \put(15,15){\small$u$}
\end{picture}
\]

Let $F$ be a flower graph with petals $P_1,\dotsc,P_\ell$ where  $\ell\geq d$.
An \demph{annotation} of $F$ is a surjection $f\colon [\ell]\twoheadrightarrow\{e_1,\dotsc,e_d\}$ onto the standard generators of $\ZZ^d$
and in each petal
$P_i$ a choice of a distinguished oriented edge $a_i\to b_i$.

We construct a $\ZZ^d$-periodic graph from an annotated flower graph $F$:
Start with the $\ZZ^d$-periodic disconnected graph $\ZZ^d\times F$, which we alter as follows.
On each copy $(\alpha,F)$ of $F$, for each petal $P_i$ remove the distinguished edge $a_i\to b_i$ and replace it by
an edge from vertex $a_i$ in  $(\alpha,F)$ to vertex $b_i$ in $(\alpha+ f(i), F)$, the translate by $f(i)$.
This gives a   \demph{$\ZZ^d$-periodic flower graph $\Gamma_F$}  which is connected as $f$ is a surjection.

\begin{Example}\label{Ex:PFG}
  We show annotations for the last three flower graphs given above.
 \[
  \begin{picture}(59,64)(0,-5)
   \put(0,0){\includegraphics{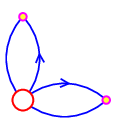}}
   \put(8,8){\small$u$}
   \put(18,45){\small$e_2$}  \put(45,18){\small$e_1$}
  \end{picture}
 \qquad
  \begin{picture}(54,64)(0,-5)
   \put(0,0){\includegraphics{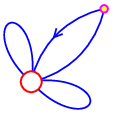}}
   \put(12,12){\small$u$}
   \put(35,52){\small$e_1$}
   \put( 3,48){\small$e_2$}  \put(45,3){\small$e_1$}
  \end{picture}
 \qquad
  \begin{picture}(65,65)(1,1)
   \put(0,0){\includegraphics{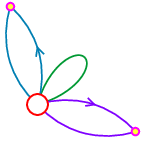}}
   \put(15,15){\small$u$}
   \put(40,44){\small$e_2$}
   \put(16,53){\small$e_1$}  \put(54,16){\small$e_2$}
  \end{picture}
 \qquad
  \raisebox{-8pt}{\includegraphics[height=80pt]{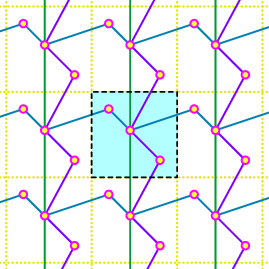}}
 \]
 The Lieb lattice is the periodic graph constructed from the first, while the second gives the graph of Figure~\ref{fig:singularHouse}.
 The graph for the third is displayed to its right.
\end{Example}

Let $F=(W,K)$ be a flower graph with vertices $W$, edges $K$ and distinguished vertex $u\in W$.
Write $L$ for the set of petals $P_i$ of $F$ that are loops, consisting of a single edge $P_i$ at $u$.
A labeling of $F$ is a pair of real-valued functions, $V\colon W\to \RR$ and $E\colon K\to \RR$.
Given a labeling, the weighted adjacency matrix \defcolor{$M$} of $F$ has 
rows and columns indexed by $W$.
For $v,w\in W$, the entry $M_{v,w}$ of $M$ is given by the following rule,
\[
    \defcolor{M_{v,w}}\ \vcentcolon=\ \left\{
    \begin{array}{rcl}
      E(v,w)&&\mbox{if } v\neq w,\\
      V(v)  &&\mbox{if } v=w\neq u,\\
      {\displaystyle V(u)+\sum_{P_i\in L} E(P_i)} &&\mbox{if } v=w= u.
    \end{array}\right.
\]

\begin{Theorem}\label{Th:FlowerGraphs}
  A periodic flower graph $\Gamma_F$ is minimally sparse.
\end{Theorem}
\begin{proof}
  Observe that each vertex of $F$ represents a unique $\ZZ^d$-orbit of vertices of $\Gamma_F$, and the same for each
  edge of $F$. 
  Thus, a labeling of $\Gamma_F$ gives a labeling of $F$ and vice-versa.
  The Floquet matrix $H(z)$ of $\Gamma_F$ may be constructed directly
  from the weighted adjacency matrix $M$ of $F$.
  Indeed, the entry $H(z)_{u,u}$ of $H(z)$ is
  \begin{equation}\label{Eq:uu-entry}
     V(u)\ +\ \sum_{P_i\in L} \left(z_{f(i)}+z_{f(i)}^{-1}\right)E(P_i)\,,
  \end{equation}
  and if $v,w\in W$ with $(v,w)\neq (u,u)$, then 
  \begin{equation}\label{Eq:H_Defn}
    H(z)_{v,w}\ \vcentcolon=\ \left\{
    \begin{array}{rcl}
      z_{f(i)}E(v,w)&&\mbox{if } v\to w\mbox{ is the directed edge in petal }P_i,\vspace{3pt}\\
      z_{f(i)}^{-1}E(v,w)&&\mbox{if } w\to v\mbox{ is the directed edge in petal }P_i,\vspace{3pt}\\
      M_{v,w}&&\mbox{otherwise.}
    \end{array}\right.
  \end{equation}

  The generic dispersion polynomial $\det(H(z){-}\lambda)$ is minimally sparse.
  Let the values of $V$ and of $E$ be indeterminates (variables).
  Note that $\det(H(z){-}\lambda)$  is obtained  by replacing each occurrence of $V(v)$ in $\det H(z)$ with $V(v){-}\lambda$.
  Thus, it suffices to show that $\det H(z)$ is minimally sparse; that the only monomials in the $z_j$ are
  single variables or their inverses.

  Writing $S_W$ for the symmetric group on the set $W$, we have
  \[
     \det H(z)\ =\ \sum_{w\in S_W} \pm \prod_{v\in W} H(z)_{v,\pi(v)}
      \ \ =\ \sum_{w\in S_W}  H(z)_\pi\,.
  \]
  The summand $H(z)_\pi$ is non-zero if and only if each entry $H(z)_{v,\pi(v)}$ for $v\in W$ is nonzero.
  When $v=\pi(v)$, so that $v$ is a fixed point of $\pi$, then $H(z)_{v,v}$ is either the constant $V(v)$ or
  else the sum~\eqref{Eq:uu-entry}, when $v=u=\pi(u)$.
  All other non-zero factors $H(z)_{v,\pi(v)}$ of $H(z)_\pi$ correspond to edges of $F$ that are not among the loops
  $L$ at $u$. 

  If  $H(z)_\pi$ is non-zero and $(v,\pi(v))$ with $v\neq \pi(v)$ is an edge in a petal $P$, then every other edge in
  $P$ occurs as $(w,\pi(w))$.
  That is, the vertices of $P$ form a cycle in $\pi$.
  This includes the distinguished vertex $u$.
  This precludes any other edge of any other petal corresponding to a factor of $H(z)_\pi$, as well as the entry $H(z)_{u,u}$.
  Similarly, if $H(z)_{u,u}$ is a factor of $H(z)_\pi$, then none of its factors correspond to edges of
  petals $P$ that are not loops in $L$. 

  These observations show that  each summand $H(z)_\pi$ has at most one factor that
  involves any variable $z_i$.
  As this factor is an entry in $H(z)$,~\eqref{Eq:uu-entry} and~\eqref{Eq:H_Defn} imply that the generic dispersion polynomial
  is minimally sparse, and completes the proof.
\end{proof}

\begin{Theorem}\label{Thm:Schroedinger}
  The spectral edges conjecture holds for Schr\"odinger operators on a $\ZZ^d$-periodic flower graph $\Gamma_F$ if and only if
  no petal of $F$ is a $2$-cycle.
\end{Theorem}   
\begin{proof}
  Fixing all edge weights $E$ to be 1 while allowing arbitrary potentials gives the family of Schr\"odinger operators on $\Gamma_F$.
  As $\Gamma_F$ is connected, generic Schr\"odinger operators on $\Gamma_F$ do not have flat bands,
  by the main result of~\cite{FaustKachkovskiy}.
  We show the dichotomy:
  \begin{enumerate}
  \item If the flower graph $F$ has a petal which is a 2-cycle, then $\Gamma_F$ has a coordinate projection whose resulting graph 
    has  bounded connected components.
  \item If no petal of $F$ is a 2-cycle, then every coordinate projection is connected.
  \end{enumerate}
  The theorem follows by Theorem~\ref{Thm:RealFamCross} and the  main result of~\cite{FaustKachkovskiy} as explained in Remark~\ref{Rem:FK}.
 
  For (1), suppose without loss of generality that the petal $P_1$ of $F$ is a 2-cycle with vertices $u$ and $w$
  ($u$ is the distinguished vertex of $F$),
  that $u\to w$ is the distinguished edge, and that $f(1)=e_d$.
  The only edges in $\Gamma_F$ involving translations of $w$ are $(w,u)$ and $(u,e_d+w)$ (and translations).
  In the Floquet matrix $H(z)$ the $(u,w)$ entry is $1+z_d$ and the $(w,u)$ entry is $1+z_d^{-1}$.
  Under the coordinate projection with $I=[d{-}1]$ and $\varepsilon_d=-1$, 
  the substitution $\zeta(z)$~\eqref{Eq:coordSubst} sets $z_d=-1$.
  Thus, in $H'(z)=H(\zeta(z))$ the entries in positions $(u,w)$ and $(w,u)$ become zero, and so $w$ and its translates are
  bounded connected components in $\Gamma'_F$.

  For (2), observe that each variable $z_i$ either occurs in the diagonal entry $H(z)_{u,u}$ in the Floquet matrix
  as $z_i+z_i^{-1}$ (if the corresponding petal is a loop) or in an entry $H(z)_{v,w}$ as $z_i$ (or $z_i^{-1}$) when $v\to w$ (or $w\to v$)
  is the distinguished edge in a petal $P_j$ with $f(j)=i$.
  In either case, if we set $z_i$ to be 1 or $-1$ in a coordinate projection $\zeta(z)$, then the corresponding entry in
  $H'(z)=H(\zeta(z))$ does not vanish and the graph $\Gamma'_F$ remains connected.
\end{proof}
\begin{Remark}
  This proof provides a general template for deciding the spectral edges conjecture for a minimally sparse graph $\Gamma$
  for a fixed choice $E$ of edge weights:
  Determine whether or not there is a coordinate projection $\Gamma'$ with bounded connected components.
  These occur only if the nonzero entries of $H(z)$ which become zero in $H'(z)=H(\zeta(z))$ correspond to edges
  whose removal creates a bounded component.
  (That is, after reordering the rows and columns, $H'(z)$  is block-diagonal with one block constant in the $z_i$.)
\end{Remark}

%
\section{Isthmus-Connected Graphs}\label{S:Isthmus}

We prove the spectral edges conjecture for another class of graphs that differ from minimally sparse graphs, but
similarly have few edges between translates of fundamental domains.

\begin{Definition}\label{D:Isthmus}
  A connected graph $G$ with an induced path whose  initial and final edges  are cut edges has an \demph{isthmus}.
  These arise from a construction.
  Let $A$ be a connected graph with $a$ vertices $v_{1-a}, v_{2-a},\dotsc,v_0$,
  $I$ be a path of length $m$ with vertices (in order) $v_1,\dotsc, v_m$, and let $B$ be a connected graph with $b$
  vertices $v_{m+1},\dotsc,v_{m+b}$. 
  Then the \demph{isthmus graph $G$} is obtained from the disjoint union $A\sqcup I \sqcup B$ by adding cut edges $(v_0,v_1)$
  and $(v_m,v_{m+1})$. 
  Either $A$ or $B$ may be empty, in which case $a$ or $b$ is zero and one or both new cut edges are not needed.
  The isthmus and any new cut edges form an induced path in $G$.
  We display three isthmus graphs, indicating their 
  components $(A,I,B)$ and  parameters $(a,m,b)$.
  \begin{equation}\label{Eq:Isthmus_Graphs}
    \raisebox{-23pt}{\begin{picture}(85,54)(0,-14)
      \put( 0,10){\includegraphics{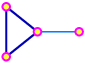}}
      \put( 7, 0){$A$}   \put(35, 0){$I$}      \put(53, 0){$B=\emptyset$}
      \put(18,-13){$(3,1,0)$}
    \end{picture}
    \qquad
    \begin{picture}(116,54)(0,-14)
        \put( 0,10){\includegraphics{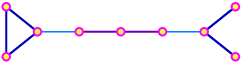}}
      \put( 7, 0){$A$}   \put(54, 0){$I$}      \put(96, 0){$B$}
      \put(37,-13){$(3,3,3)$}
    \end{picture}
    \qquad
    \begin{picture}(105,54)(0,-14)
        \put(40,15){\includegraphics{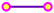}}
      \put(0, 0){$A=\emptyset$}   \put(50, 0){$I$}      \put(74, 0){$B=\emptyset$}
      \put(34,-13){$(0,2,0)$}
    \end{picture}}
  \end{equation}  

  From an isthmus graph $G=(A,I,B)$ and a function $f\colon [d]\to\{v_1,\dotsc,v_m\}$, we construct a
  $\ZZ^d$-periodic \demph{isthmus-connected graph $\Gamma$}:
  Begin with the disconnected $\ZZ^d$-periodic graph $\ZZ^d\times G$.
  For each $\alpha\in\ZZ^d$  and $j\in[d]$ add an edge between $(\alpha, f(j))$ and $(\alpha+e_j,f(j))$.
  Here are three $\ZZ^2$-periodic isthmus-connected graphs constructed from the graphs of~\eqref{Eq:Isthmus_Graphs}.
  \begin{equation}\label{Eq:threeICG}
    \raisebox{-50pt}{\includegraphics[height=110pt]{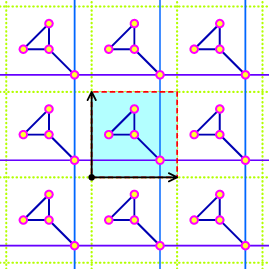}}\qquad
    \raisebox{-65pt}{\includegraphics[height=140pt]{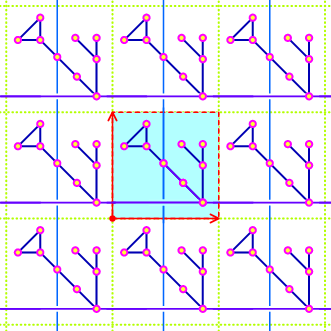}}\qquad
    \raisebox{-50pt}{\includegraphics[height=110pt]{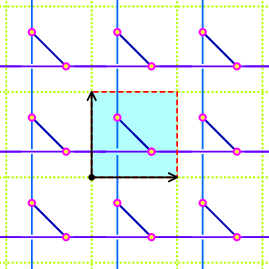}}\vspace{4pt}
  \end{equation}
  The first is a decoration of the square lattice, as studied by Schenker and Aizenman~\cite{SA}.\bigskip
\end{Definition}

\begin{Theorem}~\label{Thm:RealFamPath}
  Let $\Gamma$ be a periodic isthmus-connected graph.
  Fix nonzero edge weights.
  If the potentials are generic as in Definition~$\ref{Def:Isthmus-generic}$, then the only critical points of the Bloch variety occur at
  the corner points, and all are nondegenerate.

  The spectral edges conjecture holds for $\Gamma$ and for Schr\"odinger operators on $\Gamma$.
\end{Theorem}

We first make some definitions and clarify what is meant by generic.

\begin{Definition}\label{Def:minors}
 Let $\Gamma$ be a periodic isthmus-connected graph as in Definition~\ref{D:Isthmus}, with Floquet matrix $H(z)$.
 This has rows and columns indexed by integers from $1{-}a$ to $m{+}b$, corresponding to the vertices of the underlying
 isthmus graph $G$.
 For each index $1{-}a < s < m{+}b$, let $\defcolor{U_s(z,\lambda)}$ be the principal minor of the characteristic matrix
 $H(z)-\lambda$ formed by its rows and columns with indices {\sl less} than $s$, and let $\defcolor{L_s(z,\lambda)}$ be
 the co-principal minor formed by its rows and columns {\sl greater} than $s$.

 Define $\defcolor{P_s(z,\lambda)}\vcentcolon=\det U_s(z,\lambda)$, 
 $\defcolor{Q_s(z,\lambda)}\vcentcolon=\det L_s(z,\lambda)$, and set $P_{1-a}\vcentcolon=Q_{m+b}\vcentcolon=1$.
 Observe that if $r\leq 1$ or $s\geq m$, then no Floquet parameter appears in either $P_r$ or $Q_s$.
\end{Definition}

If $x\in\TT^d$ is a corner point, then $P_s(x,\lambda)$ is a polynomial  of degree $s{+}a{-}1$ in $\lambda$ whose coefficients
are polynomials in the potentials $V(v_i)$ for $i<s$ and weights of edges between vertices $v_i$ with $i<s$.
Similarly, $Q_s(x,\lambda)$ has degree $m+b-s$ in $\lambda$ whose coefficients are polynomials in  parameters that only involve
vertices $v_i$ with $s<i$. 
Thus, when $r\leq s$, the polynomials $P_r(x,\lambda)$ and $Q_s(x,\lambda)$ have no parameters $(V,E)$ in common.

\begin{Definition}\label{Def:Isthmus-generic}
   Parameters $(V,E)$ for a $\ZZ^d$-periodic isthmus-connected graph are \demph{generic} if 
  \begin{enumerate}
    \item No edge weight vanishes.
    \item For every corner point $x\in\TT^d$ and every pair of indices $1\leq r\leq s\leq m$, the
      polynomials $P_r(x,\lambda)$ and $Q_s(x,\lambda)$ have no common roots.
      Moreover, these polynomials have no roots in common with $D(x,\lambda)$.\qedhere
  \end{enumerate}  
\end{Definition}  

\begin{Remark}\label{Rem:resultant}
  For any corner point $x\in\TT^d$ and indices $1\leq r\leq s\leq m$,
  $P_r(x,\lambda)$ and $Q_s(x,\lambda)$ are monic polynomials in $\lambda$ and have no
  parameters from $(V,E)$ in common.
  As $\pm\prod_{i=1}^{r-1} V(v_i)$ occurs as a term in  $P_r(x,0)$ and $\pm\prod_{i=s+1}^{m+b}V(v_i)$ occurs as a
  term in $Q_s(x,0)$, the Sylvester resultant~\cite[p.~154]{CLO} of $P_r(x,\lambda)$ and $Q_s(x,\lambda)$
  is a nonzero polynomial in the parameters $(V,E)$.
  Similarly, as $D(x,\lambda)$ involves parameters $V(v_r)$ and $V(v_s)$ its resultant with either $P_r(x,\lambda)$ or
  $Q_s(x,\lambda)$ is nonzero.
  Thus, for any choice of non-zero edge weights, the set of potential values which are generic is nonempty and dense in
  $\RR^{m+a+b}$. 
\end{Remark}

The Floquet matrix of a periodic isthmus-connected graph is highly structured.
Writing \defcolor{$V_i$} for $V(v_i)$, here is a Floquet matrix for the Schr\"odinger operator on the middle graph
in~\eqref{Eq:threeICG}. 
\[
\left( \begin{array}{ccc|ccc|ccc}
    V_{-2}&  1  & 1   &  \cdot & \cdot & \cdot & \cdot & \cdot & \cdot \\
      1  &V_{-1}& 1   &  \cdot & \cdot & \cdot & \cdot & \cdot & \cdot \\
     1   &  1  &V_0  & 1 & \cdot & \cdot & \cdot & \cdot & \cdot \\  \hline
    \cdot&\cdot &  1  &V_1+z_2+z_2^{-1}  &  1 & \cdot & \cdot & \cdot & \cdot\rule{0pt}{13pt}\\
    \cdot&\cdot &\cdot & 1  &V_2 & 1 & \cdot & \cdot & \cdot\\
    \cdot&\cdot &\cdot&\cdot& 1  &V_3+z_1+z_1^{-1} & 1 &\cdot&\cdot\rule[-4pt]{0pt}{5pt}\\\hline
    \cdot & \cdot & \cdot & \cdot & \cdot & 1 & V_4  & 1 & 1\\
    \cdot & \cdot & \cdot & \cdot & \cdot & \cdot & 1 & V_5& \cdot \\
    \cdot & \cdot & \cdot & \cdot & \cdot & \cdot & 1 & \cdot & V_6 \end{array}\right)
\]
The upper left $3\times 3$ block is the weighted adjacency matrix for $A$ (a triangle), the middle block is the
triadiagonal adjacency matrix for the isthmus, and  the lower right block is for
$B$.
The 1's outside of the diagonal blocks are from the cut edges $(v_0,v_1)$ and $(v_3,v_4)$.

\begin{proof}[Proof of Theorem~\ref{Thm:RealFamPath}]
  Let $\Gamma$ be a periodic isthmus graph with nonzero edge weights and a generic potential.
  Write \defcolor{$c_r$} for the weight of the edge $(v_r,v_{r+1})$, for $r=0,\dotsc,m$ and for $j\in[d]$, write
  \defcolor{$E_j$} for the weight of the edge between the vertex $f(j)$ and its translates by $\pm e_j$.
  By construction, each Floquet parameter $z_j$ occurs in the Floquet matrix $H(z)$ as $E_j(z_j+z_j^{-1})$ and only in
  the diagonal entry of $H(z)$ corresponding to the isthmus vertex $f(j)$.

  In the vicinity of the $r$th row, the characteristic matrix $H(z)-\lambda$ has the form,
  \begin{equation}\label{Eq:row_r}
  \left( \begin{array}{c|c|c}
        U_r(z,\lambda)  &     c_{r-1}        & 0\hspace{.5pt}\quad 0 \\\hline
        0\ \  c_{r-1}& H(z)_{r,r}-\lambda & c_r \quad 0 \\\hline
        0\ \quad {0\hspace{5.5pt}}      &      c_r        &  L_r(z,\lambda) \end{array}\right)\ .
  \end{equation}

  Expanding the determinant~\eqref{Eq:row_r} along the $r$th row shows that
   \begin{equation}\label{Eq:rowExpansion}
  \pm D(z,\lambda)\ =\
  c_{r-1}^2 P_{r-1}Q_r \ -\ P_r (H(z)_{r,r}-\lambda)Q_r\ +\ c_r^2 P_r Q_{r+1}\ ,
  \end{equation}
  where if $a=0$, then $c_0=P_{-1}=0$ and if $b=0$, then $c_m=Q_{m+1}=0$.
  If $f(j)=r$, so that $z_j$ occurs in $H(z)_{r,r}$  as $E_j(z_j+ z_j^{-1})$, then none of
  $P_{r-1},P_r,Q_r$, or $Q_{r+1}$ depend upon $z_j$.
  A consequence of~\eqref{Eq:rowExpansion} is that (up to a sign) 
  \begin{equation}\label{Eq:partialDeriv}
     \frac{\partial D}{\partial z_j}(z,\lambda)\ =\
      (1-z_j^{-2}) E_j P_{f(j)}(z,\lambda)Q_{f(j)}(z,\lambda)\,,
  \end{equation}
  where  $P_{f(j)}(z,\lambda)$ and $Q_{f(j)}(z,\lambda)$ depend only upon those $z_i$ with $f(i)\neq f(j)$.

  Suppose now that \defcolor{$(x,\lambda_0)$} is a critical point of the Bloch variety.
  From~\eqref{Eq:partialDeriv}, for every $j\in[d]$, either $x_j^2=1$ or else $P_{f(j)}(x,\lambda_0) Q_{f(j)}(x,\lambda_0)=0$.
  We prove a lemma about this.

\begin{Lemma}\label{Lem:myLem}
   If $(x,\lambda_0)$ is a critical point and $x_j^2\neq 1$, then 
   $P_{f(j)}(x,\lambda_0)= Q_{f(j)}(x,\lambda_0)=0$.
\end{Lemma}
\begin{proof}
  Define $r$ by $f(j)=v_r$.
  As $\partial D/\partial z_j(x,\lambda_0)=0$, by~\eqref{Eq:partialDeriv}, one of $P_r$ or $Q_r$ vanishes at $(x,\lambda_0)$.
  Suppose that only one vanishes.
  Without loss of generality, suppose that $P_r(x,\lambda_0)=0$ and $Q_r(x,\lambda_0)\neq 0$. 
  Then from~\eqref{Eq:rowExpansion},
 \[
   0\ =\ D(x,\lambda_0)\ =\ c_{r-1}^2 P_{r-1}(x,\lambda_0)Q_r(x,\lambda_0)\,.
 \]
  Thus, $P_{r-1}(x,\lambda_0)=0$.
  To continue this downward induction, note that the principal minor $U_r$ has the following structure
 \[
      U_r\ =\   \left( \begin{array}{c|c|c}
         U_{r-2}  &     c_{r-3}        & 0 \\\hline
         c_{r-3}  & H(z)_{r-1,r-1}-\lambda & c_{r-2}  \\\hline
            0    &     c_{r-2}           & H(z)_{r,r}-\lambda  \end{array}\right)\ .
 \]
  Expanding along the last row gives $\pm P_r = (H(z)_{r,r}-\lambda)P_{r-1} - c_{r-2}^2 P_{r-2}$.
  Evaluating at the critical point $(x,\lambda_0)$ gives $P_{r-2}(x,\lambda_0)=0$.

  We may continue in this fashion until we arrive at $P_0(x,\lambda_0)=0$, which contradicts that
  the potentials are generic.
  Had we assumed $P_r(x,\lambda_0)\neq 0$ and $Q_r(x,\lambda_0)=0$, then an upward induction would arrive at the
  contradiction $Q_m(x,\lambda_0)=0$.
 \end{proof}

  We return to the proof of Theorem~\ref{Thm:RealFamPath}.
  Let $(x,\lambda_0)$ be a critical point such that $x$ is not a corner point.
  Let $r$ be the minimal index with $x_j^2\neq 1$ and $v_r=f(j)$ and let $s$ be the maximal such index.
  By Lemma~\ref{Lem:myLem}, $P_r(x,\lambda_0)=Q_s(x,\lambda_0)=0$.
  Then any Floquet parameters in $P_r$ or $Q_s$ are specialized to $\pm 1$ in $P_r(x,\lambda)$ and $Q_s(x,\lambda)$.
  Therefore, if $x^*$ is any corner point such that if $x_j^2=1$, then $x^*_j=x_j$ and we have 
  \[
     P_r(x^*,\lambda)\ =\ P_r(x,\lambda)\qquad\mbox{and} \qquad
     Q_s(x^*,\lambda)\ =\ Q_s(x,\lambda)\,.
  \]
  Thus,  $P_r(x^*,\lambda_0)=Q_s(x^*,\lambda_0)=0$ and $r\leq s$, which is a contradiction to the potential $V$ being generic.
  This shows that  critical points only occur at the corner points.

  We complete the proof by showing that each such corner critical point is nondegenerate.
  Differentiating~\eqref{Eq:partialDeriv} gives
  \[
         \frac{\partial^2 D}{\partial z_j^2}(z,\lambda)\ =\
      2z_j^{-3} E_j P_{f(j)}(z,\lambda)Q_{f(j)}(z,\lambda)\,,
  \]
  as the remaining factors are independent of $z_j$.
  On the other hand, if $i\neq j$, then 
   \begin{equation}\label{Eq:mixed}
         \frac{\partial^2 D}{\partial z_i \partial z_j}(z,\lambda)\ =\
      (1-z_j^{-2})(1-z_i^{-2}) h(z,\lambda)\,,
   \end{equation}
  for some polynomial $h$, which is nonzero only if $z_i$ occurs in the remaining factors~\eqref{Eq:partialDeriv}.

  Evaluating~\eqref{Eq:mixed} at the critical point $(x,\lambda_0)$, the mixed partial derivatives vanish.
  Hence the Hessian matrix is diagonal with entries  $\pm 2E_j P_{f(j)}(x,\lambda_0) Q_{f(j)}(x,\lambda_0)$.
  As the potential is generic and $D(x,\lambda_0)=0$, no diagonal entry vanishes.
  Thus, $(x,\lambda_0)$ is nondegenerate.
\end{proof}

%
%
\section{Parallel Extensions}\label{Sec:Parallel}

Let $a$ be a nonzero real number, and write \defcolor{$\ZZ_a$} for the infinite path with vertices from $\ZZ$ and with an edge
labeled $a$ between $n$ and $n{+}1$, for $n\in\ZZ$.
\[
  \begin{picture}(166,17)
    \put(0,0){\includegraphics{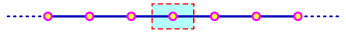}}
    \put( 31,10.5){\scriptsize$a$}  \put( 51,10.5){\scriptsize$a$}
    \put(111,10.5){\scriptsize$a$}  \put(131,10.5){\scriptsize$a$}
  \end{picture}
\]
Let $\Gamma$ be a $\ZZ^d$-periodic labeled graph.
The product $\defcolor{\Gamma_a}\vcentcolon= \ZZ_a\times\Gamma$ is a \demph{parallel extension} of $\Gamma$.
It inherits its labeling from $\ZZ_a$ and $\Gamma$.

\begin{Example}\label{Ex:parallel_extension}
Consider the $\ZZ$-periodic graph $\Gamma$, shown with a labeling and Floquet matrix.
\begin{equation}\label{Eq:linearG}
  \raisebox{-20pt}{\includegraphics{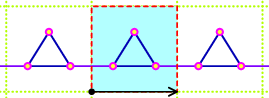}}
  \qquad
  \raisebox{-24pt}{\begin{picture}(110,55)(-16,0)
      \put(0,0){\includegraphics{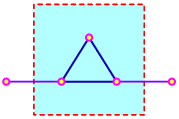}}
      \put(39,43){\small$w$}
      \put(-16, 7){\small$-1{+}v$}  \put(73, 7){\small$1{+}u$}
      \put( 26, 9){\small$u$}       \put(53, 9){\small$v$}
      \put( 39,20){\small$d$}       \put( 9,20){\small$e$}    \put(73,20){\small$e$}
      \put( 30,29){\small$b$}       \put(51,29){\small$c$} 
    \end{picture}}
  \qquad
   \left(\begin{array}{ccc}
        u  & d+ex^{-1} & b \\
      d+ex &   v      & c \\
        b  &   c      & w \end{array}\right)
\end{equation}
The only term in $\det H(z)$ involving $x$ is
$w(d+ex)(d+ex^{-1})=w(dex^{-1} + d^2+e^2 + dex)$.
Thus, $\Gamma$ is minimally supported. 
Here is its parallel extension $\Gamma_a$ and Floquet matrix.
Its vertical edges have label $a$ and the remaining edges inherit their labels from~\eqref{Eq:linearG}.
\[
  \raisebox{-50pt}{\includegraphics{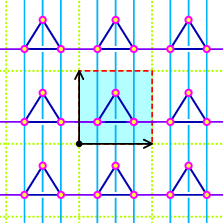}}
  \qquad 
   \left(\begin{array}{ccc}
        u+ay+ay^{-1}   & d+ex^{-1} & b \\
      d+ex &   v+ay+ay^{-1}       & c \\
        b  &   c      & w+ay+ay^{-1}  \end{array}\right)
\]
The product of diagonal entries includes monomials $y^n$ for $n=-3,\dotsc,3$, showing that the parallel extension is
not minimally supported, and that the monomials in the dispersion polynomial are quite different from those that
arise from isthmus-connected graphs.
\end{Example}

Example~\ref{Ex:parallel_extension} illustrates the general construction of the Floquet matrix $H_a(z)$ for
$\Gamma_a$ from the Floquet matrix $H(z)$ of $\Gamma$:
\[
   \mbox{add $a(z_{d+1} + z_{d+1}^{-1})$ to each diagonal entry of $H(z)$.}
\]
Consequently, the dispersion polynomial \defcolor{$D_a$} of $\Gamma_a$ is
 \begin{equation}\label{Eq:ParallelDispersion}
   D_a(z_1,\dotsc,z_d, z_{d+1}\,,\, \mu)\ =\ D(z_1,\dotsc,z_d\,,\, \mu-a(z_{d+1} + z_{d+1}^{-1}))\,,
 \end{equation}
where $D$ is the dispersion polynomial of $\Gamma$.   

For $a\neq 0$, define the map $\pi_a\colon \TT^{d+1}\times\RR \to \TT^d\times\RR$ by
\[
   \pi_a(z_1,\dotsc,z_d, z_{d+1}\,,\, \mu)\ =\ (z_1,\dotsc,z_d\,,\, \mu-a(z_{d+1} + z_{d+1}^{-1}))\,.
\]
This has fibres isomorphic to $\TT$.

\begin{Theorem}\label{Thm:Parallel_Critical}
  Suppose that $\Gamma$ is a labeled $\ZZ^d$-periodic graph and let $\Gamma_a$ with $a\neq 0$ be a parallel extension of\/ $\Gamma$.
  Write $\BV$ and $\BV\!_a$ for their Bloch varieties.
  We have
  \begin{enumerate}
    \item $\BV\!_a=\pi_a^{-1}(\BV)$.

    \item If $(x\,,\,\mu)$ is a critical point of $\BV\!_a$, then $\pi_a(x\,,\,\mu)$ is a critical point of $\BV$.

    \item If $(z\,,\,\lambda)$ is a nondegenerate critical point of $\BV$, then the critical points of $\BV\!_a$ in the fiber
           $\pi_a^{-1}(z\,,\,\lambda)$  are $(z,\pm1,\lambda\pm2a)$.

    \item If $(z\,,\,\lambda)$ is a degenerate critical point of $\BV$, then  the fiber $\pi_a^{-1}(z\,,\,\lambda)$
      consists of degenerate  critical points.

  \end{enumerate}
\end{Theorem}

We deduce that parallel extension preserves the property of spectral band functions being perfect Morse functions.

\begin{Corollary}
  If every spectral band function of a labeled graph $\Gamma$ is a perfect Morse function,
  then every spectral band function of a parallel extension of\/ $\Gamma$ is a perfect Morse function.
\end{Corollary}
\begin{proof}[Proof of Theorem~\ref{Thm:Parallel_Critical}]
  The first claim is a  consequence of~\eqref{Eq:ParallelDispersion}.
  That formula also implies that if  $i\leq d$ and $(x\,,\,\mu)\in\TT^{d+1}\times\RR$, then
  \[
    \frac{\partial D_a}{\partial z_i}(x\,,\,\mu)\ =\
    \frac{\partial D}{\partial z_i}(\pi_a(x\,,\,\mu))\,.
  \]
  The second claim follows.

  Define $\defcolor{\lambda(z_{d+1},\mu)}\vcentcolon= \mu-a(z_{d+1}+z_{d+1}^{-1})$.
  We compute
  \begin{equation}\label{Eq:ChainRule}
    \frac{\partial D_a}{\partial z_{d+1}}\ =\
    \frac{\partial D(z,\lambda(z_{d+1},\mu))}{\partial z_{d+1}}\ =\
    \frac{\partial D}{\partial\lambda}\cdot \frac{\partial \lambda}{\partial z_{d+1}}\ =\ 
    \frac{\partial D}{\partial\lambda} \cdot a(z_{d+1}^{-2}-1)\,.
  \end{equation}
  If $(z,\lambda)$ is a critical point of $\BV$, then $(z,z_{d+1},\lambda+a(z_{d+1}+z_{d+1}^{-1}))$ is a critical point
  of $\BV\!_a$ if is it a zero of $\partial D_a/\partial z_{d+1}$.
  By~\eqref{Eq:ChainRule}, we have $\partial D/\partial \lambda \cdot a (z_{d+1}^{-2}-1)=0$.

  If $(z,\lambda)$ is a nondegenerate critical point, then by~\cite[Lem.\ 5.1]{FS24},
  $\partial D/\partial \lambda \neq 0$.
  Thus, $z_{d+1}^{-2}=1$, proving (3).

  For (4), as $(z,\lambda)$ is degenerate,  $\partial D/\partial \lambda= 0$, again by~\cite[Lem.\ 5.1]{FS24}.
  Then~\eqref{Eq:ChainRule} implies that   $\partial D_a/\partial \lambda(z,z_{d+1},\lambda+a(z_{d+1}+z_{d+1}^{-1}))= 0$,
  which completes the proof.
\end{proof}

%
%

\def\cprime{$'$}

\providecommand{\bysame}{\leavevmode\hbox to3em{\hrulefill}\thinspace}
\providecommand{\MR}{\relax\ifhmode\unskip\space\fi MR }
\providecommand{\MRhref}[2]{%
  \href{http://www.ams.org/mathscinet-getitem?mr=#1}{#2}
}
\providecommand{\href}[2]{#2}


\begin{thebibliography}{10}

\bibitem{ABG}
Lior Alon, Gregory Berkolaiko, and Mark Goresky, \emph{Smooth critical points
  of eigenvectors on the torus of magnetic perturbations of graphs}, 2025, {\tt
  ArXiV.org/2502.17215}.

\bibitem{AshcroftMermin}
N.W. Ashcroft and N.D. Mermin, \emph{Solid state physics}, Cengage Learning,
  2022.

\bibitem{BCCM}
G.~Berkolaiko, Y.~Canzani, G.~Cox, and J.L. Marzuola, \emph{A local test for
  global extrema in the dispersion relation of a periodic graph}, Pure and
  Applied Analysis \textbf{4} (2022), no.~2, 257--286.

\bibitem{BerkKuch}
G.~Berkolaiko and P.~Kuchment, \emph{Introduction to quantum graphs},
  Mathematical Surveys and Monographs, vol. 186, American Mathematical Society,
  Providence, RI, 2013.

\bibitem{CLO}
David Cox, John Little, and Donal O'Shea, \emph{Ideals, varieties, and
  algorithms}, second ed., Undergraduate Texts in Mathematics, Springer-Verlag,
  New York, 1997.

\bibitem{DKS}
Ngoc Do, Peter Kuchment, and Frank Sottile, \emph{Generic properties of
  dispersion relations for discrete periodic operators}, J. Math. Phys.
  \textbf{61} (2020), no.~10, 103502, 19.

\bibitem{FaustKachkovskiy}
Matthew Faust and Ilya Kachkovskiy, \emph{Absence of flat bands for discrete
  periodic graph operators with generic potentials}, 2025, {\tt
  arXiv/2509.01927}.

\bibitem{faust2025rareflatbandsperiodic}
Matthew Faust and Wencai Liu, \emph{Rare flat bands for periodic graph
  operators}, 2025, {\tt arXiv/2503.03632}.

\bibitem{FL-G}
Matthew Faust and Jordy Lopez~Garcia, \emph{Irreducibility of the dispersion
  polynomial for periodic graphs}, SIAM J. Appl. Algebra Geom. \textbf{9}
  (2025), no.~1, 83--107.

\bibitem{FRS2025}
Matthew Faust, Jonah Robinson, and Frank Sottile, \emph{The critical point
  degree of a periodic graph (in preparation)}, 2025.

\bibitem{FS24}
Matthew Faust and Frank Sottile, \emph{Critical points of discrete periodic
  operators}, J. Spectr. Theory \textbf{14} (2024), no.~1, 1--35.

\bibitem{FLM22}
J.~Fillman, W.~Liu, and R.~Matos, \emph{Irreducibility of the {B}loch variety
  for finite-range {S}chr{\"o}dinger operators}, J. Funct. Anal. \textbf{283}
  (2022), no.~10.

\bibitem{fk18}
Nikolay Filonov and Ilya Kachkovskiy, \emph{On the structure of band edges of
  2-dimensional periodic elliptic operators}, Acta Math. \textbf{221} (2018),
  no.~1, 59--80. \MR{3877018}

\bibitem{GKT}
D.~Gieseker, H.~Kn\"{o}rrer, and E.~Trubowitz, \emph{The geometry of algebraic
  {F}ermi curves}, Perspectives in Mathematics, vol.~14, Academic Press, Inc.,
  Boston, MA, 1993.

\bibitem{Kittel}
C.~Kittel and P.~McEuen, \emph{Introduction to solid state physics}, John Wiley
  \& Sons, 2018.

\bibitem{KorotyaevSabruova14}
Evgeny Korotyaev and Natalia Saburova, \emph{Schr{\"o}dinger operators on
  periodic discrete graphs}, Journal of Mathematical Analysis and Applications
  \textbf{420} (2014), no.~1, 576--611.

\bibitem{KuchBook}
P.~Kuchment, \emph{Floquet theory for partial differential equations}, Operator
  Theory: Advances and Applications, vol.~60, Birkh\"{a}user Verlag, Basel,
  1993.

\bibitem{KuchBAMS}
\bysame, \emph{An overview of periodic elliptic operators}, Bull. Amer. Math.
  Soc. (N.S.) \textbf{53} (2016), no.~3, 343--414.

\bibitem{Liu22}
W.~Liu, \emph{Irreducibility of the {F}ermi variety for discrete periodic
  {S}chr\"{o}dinger operators and embedded eigenvalues}, Geom. Funct. Anal.
  \textbf{32} (2022), no.~1, 1--30.

\bibitem{Milnor}
J.~Milnor, \emph{Morse theory}, Annals of Mathematics Studies, No. 51,
  Princeton University Press, Princeton, NJ, 1963, Based on lecture notes by M.
  Spivak and R. Wells.

\bibitem{RS78}
Michael Reed and Barry Simon, \emph{Methods of modern mathematical physics.
  {IV}. {A}nalysis of operators}, Academic Press [Harcourt Brace Jovanovich,
  Publishers], New York-London, 1978.

\bibitem{SA}
Jeffrey~H. Schenker and Michael Aizenman, \emph{The creation of spectral gaps
  by graph decoration}, Lett. Math. Phys. \textbf{53} (2000), no.~3, 253--262.

\bibitem{AAPGO}
Stephen Shipman and Frank Sottile, \emph{Algebraic aspects of periodic graph
  operators}, 2025, {\tt ArXiv.org/2502.03659}, 24 pages.

\end{thebibliography}
\end{document}